\documentclass[11pt,reqno]{amsart}
\usepackage{graphicx}
\usepackage{amsmath,amssymb,amsthm,amsfonts,mathrsfs,stmaryrd}
\usepackage{cases,indentfirst}
\usepackage{float}
\usepackage{multicol}
\usepackage{amsmath, amsfonts, amsthm, amssymb,mathrsfs, amscd, graphicx, cite,color}

\usepackage{enumitem}

\makeatletter
\newtheorem{thm}{Theorem}[section]
\newtheorem{lem}[thm]{Lemma}
\newtheorem{cor}[thm]{Corollary}

\newtheorem{remark}[thm]{Remark}

\newcommand{\na}{D}
\newcommand{\de}{\delta}
\newcommand{\om}{\Omega}
\newcommand{\la}{\label}

\newcommand{\bnn}{\begin{eqnarray*}}
\newcommand{\enn}{\end{eqnarray*}}
\newcommand{\ba}{\begin{aligned}}
\newcommand{\ea}{\end{aligned}}
\newcommand{\be}{\begin{equation}}
\newcommand{\ee}{\end{equation}}

\def\pa{\partial}
\def\g{\gamma}

\def\ep{\varepsilon}
\def\a{\alpha}
\def\u{{\bf u}}
\def\x{\bf x}
\def\n{{\bf n}}
\def\V{{\bf v}}
\def\U{{\bf u}}
\def\a{{\bf a}}
\def\b{{\bf b}}
\def\p{{\bf p}}
\def\r{\mathbb{R}}
\def\rr{\mathbb{R}^{n}}
\def\rrr{\mathbb{R}^{n}}

\def\th{\theta}
\def\Ga{\Gamma}
\def\si{c}
\def\lap{\triangle}

\theoremstyle{definition}
\numberwithin{equation}{section}

\begin{document}

\title[Traveling wave for  gradient flow]
{\bf Existence of  traveling waves for vector valued  gradient flows}

\author[X. Chen]{Xinfu Chen}
\address{School of  Mathematics,\& Big Data Laboratory on Financial Security and Behavior(Laboratory of Philosophy and Social Science, Ministry of Education), Southwestern
University of Finance and Economics,  Chengdu 611130,  China.}
\email{chenxinfu@swufe.edu.cn}
 
\author[Z. Liang]{Zhilei Liang}
\address{School of  Mathematics, \& Big Data Laboratory on Financial Security and Behavior(Laboratory of Philosophy and Social Science, Ministry of Education),  Southwestern
University of Finance and Economics,  Chengdu 611130,  China.}
\email{liangzl@swufe.edu.cn}

\begin{abstract}
Allen--Cahn equation  is  a  fundamental continuum model  that describes  phase transitions  in multi-component mixtures.  We prove the existence of traveling waves for  vector valued  Allen--Cahn equations in the context of Ginzburg-Landau  theories; in addition,  we find the largest wave speed and  provide its  bounds from upper and below.  Our method is   based on a  variation technique and can be applied   to system of equations with a gradient flow structure. \end{abstract}

\date{\today}

\maketitle


\section{Introduction}

 One of the most celebrated microscopic continuum models for
multi--phase transition is the system of Allen--Cahn equations
\be \la{1.1} {\u}_{t}^{\ep}({\x},t) =\ep\lap {\u}^{\ep} -\ep^{-1}\na W({\u}^{\ep}),\quad {\x}\in\om\subseteq\r^{m},\,\, t>0, \ee where $\x$ and $t$ are space
and time variables, $\ep>0$   a small parameter,  ${\u}^{\ep} =(u_1^{\ep},\cdot\cdot\cdot,u_n^{\ep})$  the unknown function, $\lap$ the Laplace operator with respect to the space variables ${\x}$, 
 and $\na W(\u)=\left(\frac{\pa W(\u)}{\pa {u_1}},\cdot\cdot\cdot, \frac{\pa W(\u)}{\pa u_{n}}\right)$  the gradient of a  smooth  potential
 $W:\u\in\rrr\to W(\u)\in\r$.

 The system  \eqref{1.1}  is derived
 from a gradient flow of  the Ginzburg-Landau free energy 
 functional (see \cite{hh})  \bnn
E({\u}^{\ep}) = \int_{\om} \Big\{\frac{\ep}{2} |\nabla {\u}^{\ep}|^2 +\frac{1}{\ep}
W({\u}^{\ep})\Big\}\,{\rm d}\x,\enn where  $|\nabla {\u}|^2=\sum_{i=1}^{n}\sum_{j=1}^{m}|\frac{\pa u_{i}}{\pa x_{j}}|^{2}$, and functions $\ep^{-1}W(\u^{\ep})$ and
$\ep|\nabla{\u}^{\ep}|^2$ are, respectively,  the  potential  and
 kinetic  energy densities. In this setting, a {\it well} is
 a point of a  local minimum of $W$,
 a {\it phase} is represented by
a {\it well} of $W$ and a {\it phase $\b$ domain}  is a set
 in which ${\u}^{\ep}(\cdot,t)\approx \b$. Near an
intersection of two phase regions there is an {\it interfacial
region} where $\u^{\ep}$ changes rapidly but smoothly  from one {\it well} to
another.

 \bigskip

 For smooth ``flat'' initial data,  the diffusion $\ep \lap{\u}^{\ep}$ can be ignored in an initial short time interval, and thus \eqref{1.1}  can be approximated by   $\ep {\u}_{t}^{\ep}\approx  -\na W(\u^{\ep})$. This
 indicates that   \bnn {\u}^*({\x},t):=\lim_{\ep \searrow0}\u^{\ep}({\x},t)={\a}_i \quad\forall \,t>0,
{\x} \in  \Omega_i(t),\ i=1,\cdots,k
\enn  where each $\Omega_i$, $i=1,\cdots,k$, is  {\it phase
$\a_i$ domain}.  Clearly, to determine uniquely ${\u}^*(\cdot,t)$,
it suffices to find $\gamma(t):=\Omega\setminus(\cup_i\;
\Omega_i(t))$. Postulate that \bnn
 \gamma(t) =
\cup_{i\neq j} \gamma_{ij}(t),\quad
\gamma_{ij}(t):=\partial\Omega_i(t)\cap \partial\Omega_j(t).
\enn It is a common belief that near a point ${\bf p}$
 on the interface $\gamma_{ij}(t)$ with unit normal vector
 $\n$,
\bnn\la{appro} {\u}^\ep({\x},t) \approx
{\U}\left(\frac{({\x}-{\bf p})\cdot{\n}-c\,t}{\ep}\right)\enn
 where
$(c,\U)$ is  a traveling wave  connecting ${\a}_i$ and
${\a}_j$ with wave speed $\si$. 
 Here  by  a traveling wave that connects wells $\a$
and $\b$  with speed $\si$ and profile $\U$,  we mean  a solution $(\si,\u)\in \r\times C^{2}(\r;\rr)$
to the  problem 
\be\la{1.4} {\U}''+c\;{\U}' = \na W({\U}) \quad\hbox{in \
}\r, \qquad {\U}(-\infty)=\a,\quad
{\U}(+\infty)=\b.\ee

Traveling  solutions are special solutions of   \eqref{1.1}   which describes  uniformly translating “phase change regions”, moving with speed $\si$. Particularly, 
 if $c$ is the speed of the connection  used
 in \eqref{appro}, one has  the motion law  \bnn V(\gamma(t),p) =
c \qquad\forall p\in \gamma(t), \label{m1}\enn where $V(\gamma(t),p)$ denotes the
velocity of $\gamma(t)$ at  $p$ in the direction normal to
the interface, pointing from $\Omega_i(t)$ to $\Omega_j(t)$.  

\bigskip

Traveling waves are important in applied mathematics    because they are easy to be observed and be accurately  measured in laboratory experiments.  There  
are  standard theories for the existence of traveling waves for the scalar case.   In the celebrated paper \cite{fm},   Fife-McLeod proved the global exponential stability of traveling wave solutions of scalar equation \eqref{1.1}, by  using a  variational structure. 
Chen \cite{chen2}  proved the existence, uniqueness, and exponential stability  of monotone traveling wave solutions  for a large class of evolution equations; the  method used in \cite{chen2} relies heavily on  the comparison principle.  Chen-Guo\cite{cg} studied the existence and asymptotic stability of traveling waves to a discrete version of the quasilinear parabolic equations.  The existence of time-periodic pulsating traveling front has also been proved in the case of a bistable nonlinearity in \cite{abc}.   
 For other related results, we refer to   \cite{bn,ho,lmn,lmn1,mc,Wu} and the references  therein.
In case of  standing waves, i.e., $\si=0$, solutions to \eqref{1.4} have  been constructed in \cite{abc1} in the two dimensional case using complex analysis and regarding $(u,v)$ as  a  complexed valued functions.  In particular, multiple standing waves are constructed for a large class of potential functions.   However,  for  traveling waves ($\si\neq0$)   of semi-linear system of  equations,   the lack of comparison principle (except for cooperative   or similar kind of population growth models) prohibits  a general study of traveling wave problem.   For vector valued  equations \eqref{1.1},  Chen-Qi-Zhang \cite{cqz} established  the existence of traveling waves of a class of reaction–diffusion systems which model the pre-mixed isothermal autocatalytic chemical reaction  between two chemical species.  
 Lucia-Muratov-Novaga \cite{lmn1} studied for \eqref{1.1} the existence of a special class of traveling waves  with exponential decay in the direction of propagation,  
 and  proved the  boundedness, regularity, and some other properties. Chen-Zelati \cite{cz} studied   \eqref{1.1} on an infinite channel and proved the  existence of    traveling wave solutions   which possess a large number of oscillations.  However, little information about the limiting behavior at the ends of the cylinder is known;  see for example,  the papers \cite{fsv,m,mu}. 

\bigskip

Of our particular interest in this paper  is   the existence of traveling waves for  vector valued system \eqref{1.1}.    
For this, we  make the following assumptions:
 
 \begin{enumerate}[label=({\bf A}).]
\item {\it  $W\in C^{2}(\rr;\r);$ there exists  $\b\in \rr$ such that 
 \be\la{b0}\ba 
 W(\b)=0,\quad | \na W(\b)|=0,\quad 
 D^{2} W(\b)>0,\ea\ee
 where $\na W$ and $D^{2} W$ are respectively the gradient and Hessian of $W$; in addition,
 the sets 
 \be\la{b1} \mathbb{D}=\{{\u} \in \rr\,|\,W(\u)<0 \}\quad {\rm and}\quad \Ga:=\pa \mathbb{D} \ee
 are  non-empty and bounded. } 
 \end{enumerate}
 
 We define $\mathbb{E}$  as  the set of all equilibria in $\mathbb{D}$, i.e., 
 \be\la{b2} \mathbb{E}=\{\U\in \rr\,|\, W(\U)<0,\,\,\,\na W(\u)=\textbf{0}\}. \ee

Our main result of  this paper is the following 
\begin{thm}\la{thm}  Assume that the  function $W$ satisfies  ${\bf (A)}$.  Then
there exists $(\si,{\U})\in (0,\infty)\times C^{3}(\r;\rr)$ that satisfies 
 \be\la{6.8}\left\{\ba &\si\u'+\u''=\na W(\u)\quad{\rm in}\,\,\,\r,\\
 &\u(+\infty)=\b,\quad \lim_{x\rightarrow-\infty}W(\U(x))=w\quad {\rm for\,\,some}\,\,w<0,\\
 &  \lim_{x\rightarrow-\infty}|\na W(\U(x))|=0.
 \ea\right.\ee
Consequently,  if $\mathbb{E}=\{\a\}$ is a singleton,  then \be\la{cm} \lim_{x\rightarrow-\infty}\U(x)=\a,\ee  i.e., $(\si,{\U})$ is a solution of  \eqref{1.4}.    \end{thm}
 \begin{remark} The solution of \eqref{6.8} satisfies $$\lim_{x\rightarrow-\infty} {\rm dist}(\u(x),\,\mathbb{E})=0.$$   As an open problem,  we would like to show that  \eqref{cm} holds for some $\a\in \mathbb{E}.$  Clearly, a sufficient condition is that $\mathbb{E}$ consists of isolated points.\end{remark}
  \begin{remark}\la{rr} The condition $D^{2}W(\b)>0$  can be replaced by the weaker assumption that these exists $\de>0$ such that $W(\u)\ge0$ if $|\u-\b|\le \de;$ this assumption implies that ${\rm dist}\,(\b,\Gamma)\ge \de>0;$  see its  usage in Lemma \ref{lem3.2}.\end{remark}

 \begin{remark} Suppose $\widetilde{\b}$ is a local, but not a global, point of minimum of function $W$.  Then, working on $\widetilde{W}(\u)=W(\u)-W(\widetilde{\b})$,  we can find traveling waves connecting $\widetilde{\b}$ with a phase that has lower potential energy.  \end{remark}

\bigskip

  Our  approach  is  based on the method of variation of calculus.  The main idea is as follows:  First, for every $\si>0,$   we find a minimizer for  
  the functional
\be\la{4.1a} \ba J(\si,\,{\U})=&\int_{\r}e^{\si x}\left(\frac{1}{2}|{\U}'|^{2}+W({\U})\right) {\rm d}x \ea\ee
in the  admissible  set  $\mathcal{A}$  defined by \be\la{4.1b} \ba \mathcal{A}:=\left\{{\U}\in H^{1}_{{\rm loc}}(\r;\rr)\,\,
\Big|  \ba\U(+\infty)=\b,\,\,{\U}(0)\in \Gamma,\,\,W({\U})\ge 0\,\,{\rm in}\,\,(0,+\infty)\ea\right\}\ea \ee
where  $\Gamma=\pa \mathbb{D}=\pa \{{\u} \in \rr\,:\,W(\u)<0 \}.$

Next,  we define  the minimum energy function  $\g(\si)$ by  $$\g(\si):=\inf_{\mathcal{A}} J(\si,\cdot)\quad \forall \si>0.$$ As we shall see later,  a minimizer $\u$ of $J(\si,\cdot)$ in $\mathcal{A}$ satisfies $\si\U'+\U''-\na W(\U)=0$ in $\r$ if and only if $\g(\si)=0$. Hence we search for the root of $\g(\si)=0$.  For this,   we show that $\g$ is continuous, strictly increasing and  negative when $\si>0$ is  small, and positive when $\si$ is  large. This tells that there exists a unique  $\si^{*}$ such that $\g(\si^{*})=0$.  If $\U^{*}$ is a minimizer of $J(\si^{*},\cdot)$ in $\mathcal{A},$ then by the method of variation of calculus, we   show that $(\si^{*},\U^{*})$ is a traveling wave. 
This method is quite straightforward and we expect it can be  extended for other systems, e.g., those that contain a drift term, or the traveling wave problem in \cite{lmn1}.
\begin{remark}  We would like to mention that  the paper \cite{lmn1} by  Lucia-Muratov-Novaga also used  the functional similar to that  in \eqref{4.1a} for the existence of traveling waves.  To fix the translation invariance, they used  the  constraint $\int_{\r}e^{\si x}|\u'(x)|^{2}{\rm d}x=1$ instead of the admissible set $\mathcal{A}.$   In comparison with the results in \cite{lmn1}, our assumption $({\bf A})$ is much simpler, and moreover, our analysis  seems more straightforward and elegant. \end{remark}

\begin{remark} 
We remark that the root $\si^{*}$ of $\g(\cdot)=0$  is the largest speed of all traveling waves.  See Theorem \ref{thm4.1} and the example in Section 5.  Similar conclusion also appears in \cite{lmn1}.\end{remark}

   {\bf Notations}: We denote a vector by   $\a=(a_{1},\cdot\cdot\cdot,a_{n})$   and   a vector-valued function by $\u=(u_{1},\cdot\cdot\cdot,u_{n})$.     The characteristic   function of a set $A$ is denoted by $\textbf{1}_{A}$, i.e., $\textbf{1}_{A}(x)=1$ if $x\in A$, and   $\textbf{1}_{A}(x)=0$ if $x\notin A$.    We use $\overline{w}\wedge 0=\min\{\overline{w},\,0\},$  
and  write $f=O(1)g$ as $x\rightarrow +\infty$ provided that  there is a constant $M$ such that $|f(x)|\le M|g(x)|$ for all $x$ sufficiently large.     We denote by  $C^{\infty}(\om;\r^{n})$   the set of all
smooth   functions  from $\om$ to $\rrr$, and     $C_{0}^{\infty}$  denote these functions in $C^{\infty}$ with compact  support.   

\bigskip

The rest of this  paper is arranged as follows.
In Section 2, we first  estimate the size of minimum energy  $\g(\si),$ and then   show the existence of a minimizer  of $J(\si,\cdot)$ in $\mathcal{A}$ for every $\si>0.$   In Section 3, we study the minimizers. In particular,  we study its asymptotic behavior as $x\rightarrow-\infty.$ In Section 4,  we prove that  $\g(\si)$ is continuous and monotone.   In Section 5, we prove Theorem \ref{thm}. Finally, in Section 6  we present an example illustrating   the non-existence and non-uniqueness of traveling waves.

\bigskip

In the sequel, we shall always assume that $W$ satisfies $({\bf A}).$ 

 \section{Variation Method and Energy Estimates}
 In this section, we first  estimate the size of the minimum energy $\g(\si).$ Then we show that, for each $\si>0$, $J(\si,\cdot)$ admits at least a minimizer in $\mathcal{A}$. 
 
  \subsection{Energy Estimates}
 We provide some  estimates on the size of  minimum energy.
 \begin{lem} \la{lem3.2}  For each $\si>0$ define $\g(\si):=\inf_{\mathcal{A}} J(\si,\cdot).$ Then 
  \be\la{5.10}  \frac{\si d^{2}}{2}-\frac{m}{\si} \le  \g(\si)<   
 \frac{ 1}{\si} \left(\left(\frac{|\b-\a|^{2}}{2} +M\right)(e^{\si}-1)-  m e^{-\si}\right),\ee
 where  $\a\in \mathbb{D}$ is  a point satisfying $W(\a)=\inf_{\rr}W(\cdot)$,\,  $m=-W(\a),$  
 \be\la{6.23} M=\max_{x\in [0,1]}W(\a+x(\b-\a))\quad {\rm and}\quad d= \inf_{W(\U)<0}|\U-\b|.\ee
     \end{lem}
 \begin{remark} The assumption \eqref{b0} implies that there exists some  $\de>0$ such that  $W(\U)>0$ if $0<|\U-\b|<\de.$ Hence, $d$ in \eqref{6.23}  is positive;  see  Remark \ref{rr}.\end{remark}
 
 \begin{proof}  Consider the function
\bnn\la{5.5}\U_{0}(x):=\left\{\ba &\b&{\rm if}\,\,\,&x\ge 1,\\
&\a+x(\b-\a)\quad&{\rm if}\,\,\,&0\le x\le 1,\\
&\a&{\rm if}\,\,\,&x\le 0.\ea\right. \enn
Then, there exists a point $x_{0}\in (0,1)$ such that $W(\U_{0}(x))\ge0$ for all $x\ge x_{0}$ and $\U_{0}(x_{0})\in \Ga$. Set  $\U_{x_{0}}(\cdot)=\U_{0}(x_{0}+\cdot)$. Then,   $\U_{x_{0}}\in \mathcal{A}$.  
A simple computation shows  that 
 \bnn\la{5.7}\ba \g(\si)&\le J(\si,\U_{x_{0}})
 =e^{-\si x_{0}} J(\si,\U_{0})\\
 &=e^{-\si x_{0}}\left(\int_{-\infty}^{0}e^{\si x}W(\a){\rm d}x
 +\int_{0}^{1}e^{\si x}\left(\frac{|\b-\a|^{2}}{2} +W(\a+x(\b-\a))\right){\rm d}x\right)\\
  &< -\frac{m e^{-\si}}{\si}
 + \left(\frac{|\b-\a|^{2}}{2} +M\right)\frac{e^{\si}-1}{\si}.
 \ea\enn
 This proves the second inequality in \eqref{5.10}.
 
 Next, let   $d>0$ be defined  in \eqref{6.23}.   Let $\U\in \mathcal{A}.$ Then $\U(0)\in \Ga=\pa\mathbb{D}.$ Hence,
  \bnn\ba  d^{2} &\le  |{\U}(0)-\b|^{2}=\left|\int^{\infty}_{0}\U'(s){\rm d}s\right|^{2}\le
  \frac{1}{\si}\int^{\infty}_{0}e^{\si s}|\U'(s)|^{2}{\rm d}s.\ea\enn
Consequently,  since $W(\U)\ge0$ in $[0,+\infty)$ and $W(\U)\ge -m$ in $(-\infty,0],$ we get 
\bnn\la{5.9}\ba J(\si,\U) 
\ge \frac{1}{2}\int_{0}^{\infty}e^{\si x}|\U'|^{2}{\rm d}x +\int_{-\infty}^{0}e^{\si x} W(\U){\rm d}x\ge \frac{\si d^{2}}{2}-\frac{m}{\si}. \ea \enn
Taking the infimum of $\U$ over $\mathcal{A}$, we obtain the first inequality in \eqref{5.10}.
This completes the  proof  of Lemma \ref{lem3.2}.
 \end{proof}
 
 From \eqref{5.10} we immediately   obtain the following 
 \begin{cor}\la{cor1}   There exist  positive constants $c_{1}$ and $c_{2}$ such that  $\g(\si)>0$ if $\si>\si_{2}$ and  $\g(\si)<0$ if $0<\si<\si_{1}$.  In particular,  if $\g(\si^{*})=0$, then $\si^{*}$ is bounded by 
 \be\la{cm1} \ln \frac{1+\sqrt{1+\frac{8m}{|\b-\a|^{2}+2M}}}{2}< \si^{*}\le \frac{\sqrt{2m}}{d}.\ee
   \end{cor}
 
 \subsection{Existence of a Minimizer}
 
 Now we are ready to prove the following 
\begin{thm} \la{lem3.0a}  For every $\si>0$,  the functional $J(\si,\,\cdot)$ defined in \eqref{4.1a}  admits at least  one minimizer in $\mathcal{A}$ defined in \eqref{4.1b}, i.e., there exists  $\U\in \mathcal{A}$ such that 
\be\la{a23} J(\si,\,\U)=\g(\si):=\inf_{\mathcal{A}}J(\si,\cdot).\ee \end{thm}
\begin{proof}  In view  of \eqref{5.10},    there exists  a sequence $\{{\U}_{n}\}_{n=1}^{\infty}$ in $\mathcal{A}$ such that
\bnn\la{4.3a}  \g(\si)\le J(\si,\,{\U}_{n})\le  \g(\si)+\frac{1}{n}\quad \forall n\ge 1.\enn 

We complete  the proof in two steps. 

{\it Step 1.\,Compactness of  $\{{\U}_{n}\}_{n=1}^{\infty}$.}

Denote by $\textbf{1}_{A}$ the characteristic function of the  set $A$.  Then, for $\u\in \mathcal{A},$ 
 \bnn\ba   \int_{\r} e^{\si x}\frac{|{\U}_{n}'|^{2}}{2} {\rm d}x
 &\le  \int_{\r} e^{\si x}\left(\frac{|{\U}_{n}'|^{2}}{2}+W({\U}_{n})+m\textbf{1}_{(-\infty,0)}\right) {\rm d}x \\
  &= J(\si,\,{\U}_{n})+\frac{m}{\si}\le  \g(\si)+\frac{1}{n}+\frac{m}{\si}.\ea\enn
Consequently, for any $x\in \r,$
 \begin{align}\la{5.3}   |{\U}_{n}(x)-\b|^{2}  &=\left|\int^{\infty}_{x}\U_{n}'(s){\rm d}s\right|^{2}\nonumber\le \left(\int^{\infty}_{x}e^{-\si s}{\rm d}s\right) \left(\int^{\infty}_{x}e^{\si s}|\U_{n}'(s)|^{2}{\rm d}s\right)\nonumber\\
 &\le \frac{2 e^{-\si x}}{\si}\left(\g(c)+\frac{1}{n}+\frac{m}{\si}\right).\end{align}
 This implies  that  
 \bnn\la{4.8a} \sup_{n\ge1}\,\|{\U}_{n}\|_{H^{1}((-R,R))}<\infty\quad \forall\,\,R>0.\enn
Hence, there exists a function  $\U\in H^{1}_{{\rm loc}}(\r;\rr)$ such that,  by replacing   $\{{\U}_{n}\}_{n=1}^{\infty}$ by  a subsequence if necessary,   for every $R>0,$
 as $n\rightarrow\infty,$
 \be\la{5.1} \U_{n}\,-\!\!\!\rightharpoonup \U\,\,\,{\rm weakly\,\, in}\,\,\,H^{1}((-R,R))\quad {\rm and}\quad  \U_{n}  \longrightarrow \U\,\,\,{\rm in}\,\,\,C([-R,R]).\ee
 
 {\it Step 2.\, Existence of a Minimizer in $\mathcal{A}$.}
 
 We show that   $\U$ in \eqref{5.1} is a minimizer of $J(\si,\cdot)$ in $\mathcal{A}$.
 For every $R>0$,
 \bnn\ba &\int_{-R}^{R}e^{\si x}\left(\frac{1}{2}|\U|^{2}+W(\U)+m\textbf{1}_{(-\infty,0)}\right){\rm d} x\\
 &\qquad\le \liminf_{n\rightarrow\infty} \int_{-R}^{R}e^{\si x}\left(\frac{1}{2}|\U_{n}|^{2}+W(\U_{n})+m\textbf{1}_{(-\infty,0)}\right){\rm d} x\\
 &\qquad\le \liminf_{n\rightarrow\infty} \int_{\r}e^{\si x}\left(\frac{1}{2}|\U_{n}|^{2}+W(\U_{n})+m\textbf{1}_{(-\infty,0)}\right){\rm d} x\\
 &\qquad= \g(\si) +\frac{m}{\si},
\ea\enn
which yields, by sending $R\rightarrow\infty$,
\bnn \int_{\r}e^{\si x}\left(\frac{1}{2}|\U|^{2}+W(\U)+m\textbf{1}_{(-\infty,0)}\right){\rm d} x \le \g(\si)+ \frac{m}{\si}.\enn
Hence, $ J(\si,\,\U)\le  \g(\si).$
 
Next we  show  that $\U\in \mathcal{A}.$    In fact,   $W(\U)\ge 0$ in $[0,+\infty)$ because  $\u_{n}\in \mathcal{A}$ implies that $W(\u_{n})\ge0$  in $[0,+\infty)$.   In addition, since $\U_{n}(0)\in \Ga$ and the set $\Ga$ is closed,  it follows that   $\U(0)\in \Ga.$
 Finally,
it follows from  \eqref{5.3} and \eqref{5.1} that 
 \be\la{5.4}\ba   |{\U}(x)-\b|^{2}
 \le \frac{2 e^{-\si x}}{\si}\left(\g(\si)+\frac{m}{\si}\right)\quad\forall x\in \r.\ea\ee
 This implies  $ \lim_{x\rightarrow+\infty}\U(x)=\b.$  Thus, $\U\in \mathcal{A}$. Consequently, $\U$ is a minimizer of $J(\si,\cdot)$  in $\mathcal{A}.$ 
This  completes the  proof of Theorem \ref{lem3.0a}.
 \end{proof}

 We will seek for the existence of  $\si$ such that $\g(\si)=0$. For this purpose, it suffices to show that $\g(\cdot)$ is continuous with respect to $\si.$  To do this, we need certain estimates about the minimizers. 
 
 \section{Certain Properties of Minimizers}
 In this section, we establish certain properties of a minimizer. This will be used to show that $\g(\cdot)$ is strictly monotonic and Lipschitz continuous.
 \subsection{First Variation of the Functional at  Minimizer}
 \begin{lem}\la{lem3.1}  For every  fixed $\si>0$,  a  minimizer  $\u$   of  $J(\si,\cdot)$ in $\mathcal{A}$ satisfies 
 \be\la{6.12}\left\{\ba& \si\U'+\U''=\na W(\U)\quad {\rm in}\,\,(-\infty,\,0],\\
&\U(0)\in \Ga, \quad \U\in C^{3}((-\infty,0];\rr).\ea\right.\ee
  \end{lem}
  \begin{proof}  Let $\u$   be a minimizer of  $J(\si,\cdot)$ in $ \mathcal{A}$.    Let $\varphi\in C_{0}^{\infty}((-\infty,0);\rr)$ be an arbitrary function.  Extend $\varphi$ to $[0,+\infty)$ by setting $\varphi(x)={\bf 0}$ for $x>0.$   Then we  have  $\U+t\varphi\in \mathcal{A}$ for all $t\in (-1,1)$.  Thus, 
  the function $t\in (-1,1)\mapsto J(\si,\U+t\varphi)$ attains its minimum at  $t=0.$ Hence,
    \bnn\ba 0&=\left\langle \frac{\de J(\si,\U)}{\de \U},\,\varphi\right\rangle: =\frac{d J(\si,\U+t\varphi)}{dt}\Big|_{t=0}=\int_{\r}e^{\si x}\left(\U'\cdot \varphi'+\na W(\U)\cdot\varphi\right){\rm d}x.
    \ea\enn
    Thus, $\U$ is an $H_{{\rm loc}}^{1}(\r;\rr)$ weak solution of the  equation
    \bnn -\left(e^{\si x}\U'\right)'+ e^{\si x}\na W(\U)=0\quad {\rm in}\,\,\,(-\infty,0).\enn
    By a standard   estimate \cite{evans}, we conclude that $\U$ is  a classical solution of \eqref{6.12}. 
   \end{proof}


 \begin{lem}\la{lem5.1} Let   $\u$ be a     minimizer   of  $J(\si,\cdot)$ in $ \mathcal{A}$.  
 Then,   $\U$ satisfies  
 \be\la{6.13} 
  \si\U'+\U''=\na W(\U)\quad {\rm in}\,\,I(\U):=\{x>0\,:\,\U(x)\notin \Ga\}.
\ee \end{lem}
   \begin{proof}  Since $\U$ is continuous and $\Ga$ is closed, we see that the set $I(\U)$ is open. 
We    express $I=\cup_{i=1}^{\infty}(a_{i},b_{i}).$  Let $\varphi\in C_{0}^{\infty}((a_{i},b_{i});\rr)$ be an arbitrary function. Set $\varphi(x)={\bf 0}$ if $x\notin (a_{i},b_{i}).$  We can find a  small number $t_{0}>0$  to satisfy 
\bnn \U+t\varphi\in \mathcal{A}\quad \forall  t\in (-t_{0},t_{0}).\enn
Hence, by the first variation argument, we get 
$\si\U'+\U''=\na W(\U)$ in  $(a_{i},b_{i})$
for every $ i=1,2,\cdot\cdot\cdot.$ The proof is completed.
   \end{proof}

\begin{lem} \la{lem7.1} Let $\u$ be a minimizer of $J(c,\cdot)$ in $\mathcal{A}$.   For any $\si>0$, there exists a constant $\lambda>\si$ such that  
any minimizer $\U$ of $J(\si,\cdot)$ in $\mathcal{A}$ satisfies 
  \be\la{6.24}  |\U'(x)|+|\U''(x)|+|\u(x)-\b| = O(1) e^{-\lambda x}\quad  {\rm as}\quad x\rightarrow+\infty.\ee
  Consequently,     $|W(\U(x))|=O(1)|\u(x)-\b|^{2}=O(1) e^{-2\lambda x}$ as $x\rightarrow +\infty.$
   \end{lem}
   \begin{proof}  Define $ x_{*}:=\max\{x\ge0\,|\,\U(x)\in \Ga\}$.
   By  \eqref{5.4},  $ x_{*}$ is finite.  By  \eqref{6.13},  \bnn\la{6.25}  \si\U'+\U''=\na W(\U)\quad {\rm in}\,\,(x_{*},+\infty),\quad \lim_{x\rightarrow+\infty}{\U}(x)=\b.\enn
  Let $\mu>0$ be the smallest eigenvalue of $D^{2}W(\b).$ By a  standard theory of   ordinary differential equations (e.g., \cite{hart}), we can show that \eqref{6.24} holds 
for every  $\lambda\in (0,\Lambda)$ (or $\lambda=\Lambda$ if $W\in C^{2+\sigma}$ for some $\sigma>0$), where $\Lambda$ is the positive root of the characteristic equation $\Lambda^{2}-\si \Lambda -\mu=0$, i.e., $$\Lambda=\frac{\si+\sqrt{\si^{2}+4\mu}}{2}>\si.$$  
  The desired  assertion thus follows. 
 \end{proof}
 
\subsection{The Asymptotic  Behavior of a Minimizer as $x\rightarrow-\infty$}
 
We shall  prove  the existence of the  limit of  $W(\U(x))$ as $x\rightarrow-\infty.$
 For this we define 
\be\ba\la{4.1}\underline{w}= \liminf_{x\rightarrow-\infty} W(\U(x))\quad{\rm and }\quad \overline{w}=\limsup_{x\rightarrow-\infty} W(\U(x)).\ea\ee
\begin{lem}\la{lem4.1}  Let $\u$ be a minimizer of $J(c,\cdot)$ in $\mathcal{A}$.  Define  $ \underline{w}$  as in \eqref{4.1}. Then, \be\la{9.1} W(\U(x))\ge \underline{w}\quad \forall x\le 0.\ee \end{lem}
\begin{proof} 
If the assertion is not true, there exists some $x_{0}\le 0$ such that 
\bnn\la{8.1} W(\U(x_{0}))=\min_{x\in (-\infty,0]}W(\U(x))<\underline{w}.\enn
By setting $\V(x)=\U(x) \textbf{1}_{(x_{0},\infty)}(x)+\U(x_{0})\textbf{1}_{(-\infty,x_{0}]}(x),$
one can easily verify that 
$J(\si,\V)<J(\si,\U)$. This contradicts to the fact that $\U$ is a minimizer of $J(\sigma,\cdot).$  The proof of Lemma \ref{lem4.1} is completed. 
\end{proof}

\begin{lem}\la{lem4.2}    Let $\u$ be a minimizer of $J(c,\cdot)$ in $\mathcal{A}$.   
Then  ${\displaystyle \lim_{x\rightarrow-\infty}W(\U(x))}$ exists.    \end{lem}
\begin{proof} Let $\underline{w}$ and $\overline{w}$ be defined as  in \eqref{4.1}.  It suffices to show that  $\underline{w}=\overline{w}.$ 
First by taking  $x=0$ in \eqref{9.1} we  find  that $\underline{w}\le  W(\U(0))=0.$
We shall   prove  the assertion $\overline{w}= \underline{w}$  in two cases:  (1). $\underline{w}=0;$ (2). $\underline{w}<0.$

{\it Case (1). $\underline{w}=0.$}  Set $\V(x)=\U(x) \textbf{1}_{(0,\infty)}+\U(0)\textbf{1}_{(-\infty,0]}.$
Then from \eqref{9.1} we see that   $J(\si,\V)\le J(\si,\U)$, and $J(\si,\V)< J(\si,\U)$ if  $\V \neq\U$.   Since $\U$ is a minimizer, it  follows that  $\V \equiv\U$. Therefore, $\U(x)\equiv \U(0)$ for all $x<0$,  so $\overline{w}= \underline{w}=0$.    

{\it Case (2). $\underline{w}<0.$} We prove $\overline{w}=\underline{w}$    by a contradiction argument.  Assume that $\overline{w}>\underline{w}.$ Then, as $x\rightarrow-\infty,$ $W(\U(x))$ oscilates between  $\overline{w}$ and $\underline{w}$ infinitely times. Hence there exists a sequence  $\{x_{n}\}_{n=1}^{\infty}$  in $(-\infty,0]$ such that 
\be\la{4.5} \lim_{n\rightarrow\infty}x_{n}=-\infty\quad {\rm and}\quad  W(\U(x_{n}))=\frac{\underline{w}+3\overline{w}\wedge 0}{4}\quad \forall n\ge1,\ee where $\overline{w}\wedge 0=\min\{\overline{w},0\}.$  
Let  $\{y_{n}\}_{n=1}^{\infty}$ be a sequence defined by 
\bnn\la{4.6} y_{n}=\min\left\{y\ge x_{n} \,:\,  W(\U(y))=\min_{ [x_{n},0]}\,W(\U(\cdot))\right\}\quad \forall n\ge1.\enn
It is clear  that $\lim_{n\rightarrow\infty}W(\U(y_{n}))=\underline{w}$.  By assumption ${\bf (A)}$, the set $ \mathbb{D}=\{{\U} \in \rr\,|\,W(\U)<0 \}$ is bounded. Hence, by selecting a   subsequence  if necessary, we can find $\underline{\U}\in \mathbb{D}$ such that \be\la{4.7}   \lim_{n\rightarrow \infty}\U(y_{n})=\underline{\U}\quad {\rm and}\quad W(\underline{\U})=\underline{w}.\ee
Since $\overline{w}\wedge 0>\underline{w},$ we might as well assume that 
\be\la{4.8a}\ba & \sup_{n\ge1}\,\sup_{\th\in[0,1]}W(\th\underline{\U}+(1-\th)\U(y_{n}))\le \frac{3\underline{w}+\overline{w}\wedge 0}{4}.\ea\ee
Let $z_{n}=\min\left\{x\ge x_{n}\,:\,W(\U(x))=\frac{\underline{w}+\overline{w}\wedge 0}{2}\right\}.$ Then, from \eqref{4.5} and \eqref{4.8a} we have $x_{n}<z_{n}<y_{n},$ and moreover
\be\la{4.10} W(\U(z_{n}))=\frac{\underline{w}+\overline{w}\wedge 0}{2}<W(\U(x))\quad \forall  x\in [x_{n},z_{n}).\ee
Next, we  define
\be\la{4.14}\ba \V_{n}(x)=\left\{\ba&\U(x)&{\rm if}&\,\, x>y_{n},\\
&\U(y_{n})&{\rm if}&\,\, z_{n}\le x\le y_{n},\\
&\U(y_{n})+\frac{\underline{\U}-\U(y_{n})}{\int_{x_{n}}^{z_{n}} |\U'(s)|ds}\,\int_{x}^{z_{n}} |\U'(s)|ds&{\rm if}& \,\, x_{n}\le x\le z_{n},\\
&\underline{\U}&{\rm if}&\,\,x\le x_{n}.\ea\right.\ea\ee
Recall that $W(\U(x))\ge \underline{w}=W(\underline{\U})$
for all $x\le0$ and $W(\U(y_{n}))= \min_{[x_{n},0]}W(\U(\cdot))\le W(\U(x))$ for $x\in[x_{n},0].$ We see that 
\be\la{4.9}\ba &J(\si,\U)-J(\si,\V_{n})\\
&>\int_{x_{n}}^{z_{n}}e^{\si x}\left(\frac{1}{2}(|\U'(x)|^{2}-|\V_{n}(x)|^{2})+W(\U(x))-W(\V_{n}(x))\right){\rm d}x.\ea\ee

We shall  estimate \eqref{4.9} as follows.

(i).  Observe from \eqref{4.14} that when $x\in (x_{n},z_{n})$
\bnn \ba |\V_{n}'(x)|&=\frac{|\underline{\U}-\U(y_{n})|}{\int_{x_{n}}^{z_{n}} |\U'(s)|ds}\,|\U'(x)|\le \frac{|\underline{\U}-\U(y_{n})|}{|\U(z_{n})-\U(x_{n})|}\,|\U'(x)|.
\ea\enn From  \eqref{4.7},  
 \eqref{4.5}, \eqref{4.10}, 
and the assumption that $\mathbb{D}$ is bounded,  we derive  that 
\bnn\la{4.12}  |\V_{n}'(x)| \le |\U'(x)| \quad {\rm for\,\,all} \,\,\,x\in (x_{n},z_{n})\,\,{\rm when}\,\,n \,\,{\rm is\,\,large}.\enn

(ii). For all $x\in [x_{n},z_{n}]$ 
\bnn W(\U(x))\ge W(\U(z_{n}))=\frac{\underline{w}+\overline{w}\wedge 0}{2}\enn
and 
\bnn W(\V_{n}(x))\le  \sup_{\th\in[0,1]}W(\th\underline{\U}+(1-\th)\U(y_{n}))\le \frac{3\underline{w}+\overline{w}\wedge 0}{4}.\enn
We deduce that 
\bnn\la{4.17} W(\V_{n}(x))< W(\U(x))\quad {\rm for\,\,all} \,\,\,x\in [x_{n},z_{n}]\,\,{\rm when}\,\,n \,\,{\rm is\,\,large}.\enn
It then follows  from \eqref{4.9} that, $J(\si,\U)-J(\si,\V_{n})>0$ for large $n.$
This leads to a  contradiction as $\U$ is the minimizer. Thus $\underline{w}=\overline{w}.$   
The proof of Lemma \ref{lem4.2} is completed.\end{proof}

 \begin{lem}\la{lem4.3}    Let $\u$ be a minimizer of $J(c,\cdot)$ in $\mathcal{A}$.   Then 
 \be\la{4.19} \lim_{x\rightarrow-\infty}\left(|\na W(\U(x))|+|\U'(x)|\right)=0.\ee
\end{lem}
\begin{proof}  Remember that $w=\lim_{x\rightarrow-\infty}W(\U(x))$ exists. If $w=0$, we have $\U(\cdot)\equiv\U(0)$ on $(-\infty,0].$ If $w<0,$ we find that $\U$ is bounded on $(-\infty,0].$  By elliptic estimates, we see from  \eqref{6.12} that 
\be\la{a21} \|\U\|_{C^{3}((-\infty,0])}<\infty.\ee
Now let   $\{x_{n}\}_{n=1}^{\infty}$ be a sequence   such that
$\lim_{n\rightarrow\infty} x_{n}=-\infty$
and \bnn\la{4.21}  \lim_{n\rightarrow\infty}\left( |\na W(\U(x_{n}))|+|\U'(x_{n})|\right)=\limsup_{x\rightarrow-\infty}\left( |\na W(\U(x))|+|\U'(x)|\right).\enn
We assume that,  by replacing $\{x_{n}\}_{n=1}^{\infty}$ by a subsequence if necessary,    for some $\underline{\U}$ and $\p\in \rr,$
\bnn\la{4.22a} (\U(x_{n}),\,\U'(x_{n}))\longrightarrow (\underline{\U},\,\p)\quad {\rm as}\quad n\rightarrow\infty.\enn
Define  $ \V_{n}(x)=\U(x_{n}+x)$.  From  \eqref{6.12}  and the continuous dependence of  solutions on the initial values for ODE systems,  the limit $ \V(x)=\lim_{n\rightarrow\infty}\V_{n}(x) $ exists.  In addition,
 from  \eqref{6.12} and \eqref{a21}, $\V$ is a solution of the following 
\be\la{4.22}\left\{\ba& \si\V'+\V''=\na W(\V)\quad {\rm in}\,\,\r,\\
&\V(0)=\underline{\U},\quad \V'(0)=\p,\quad \|\V\|_{C^{3}(\r)}<\infty.\ea\right.\ee
Recall from Lemma \ref{lem4.2} that $\lim_{x\rightarrow-\infty}W(\U(x))=w.$   We have 
\bnn\la{8.3}W(\V(x))=\lim_{n\rightarrow\infty}W(\U(x_{n}+x))=w\quad \forall x\in\r. \enn
This implies that $W(\V(\cdot))$ is a constant function, and consequently, 
\bnn 0=\frac{{\rm d}}{{\rm d}x}W(\V(x))=\na W(\V(x))\cdot \V'(x)\quad \forall x\in\r. \enn
 By this and \eqref{4.22} we compute 
 \bnn\ba 0&=e^{2\si x}\na W(\V)\cdot \V'(x)=e^{2\si x}\left(\si\V'+\V''\right)\cdot \V'(x)=\left(\frac{|\V'|^{2}}{2} e^{2\si x}\right)',\ea\enn
 and thus, there  exists a constant  $A\ge 0$ such  that 
 $|\V'(x)|=A e^{-\si x}$. Since  $\V'$ is bounded in $\r,$ we must have  $A\equiv0$, and thus, $\V$ is a constant vector function.   Consequently, $\p=\textbf{0}$ and $\na W(\underline{\U})=\si \V'(0)+\V''(0)=\textbf{0}$. Hence,
 $$\limsup_{x\rightarrow-\infty}\left(|\na W(\U(x))|+|\U'(x)|\right)=|\na W(\underline{\U})|+|\p|=0.$$
    The proof of Lemma \ref{lem4.3} is  completed.
\end{proof}

\section{Monotonicity and Continuity of $\g(c)$}
  \subsection{Variation of $J(\si,\U)$ in $\si$}  
  To prove the monotinicity and continuity of $\g,$ we first study the function $J(a,\U)-J(\si,\U)$ where $\U$ is a minimizer of $J(\si,\cdot)$ in $\mathcal{A}.$  
   \begin{lem}\la{lem7.2} Let $\si>0$ and $\U$ be a minimizer of $J(\si, \cdot)$ in $\mathcal{A}$.   The  following holds:
  \begin{enumerate}
\item $ W(\U),\,\,\, |\U'| \in C^{\frac{1}{2}}(\r\setminus \{0\}),\,\, \left( \frac{1}{2}|\U'|^{2}-W(\U)\right)\in C^{\frac{3}{2}}(\r\setminus \{0\}),$ and 
\be\la{a1} \left( \frac{1}{2}|\U'|^{2}-W(\U)\right)' +\si |\U'|^{2}=0\quad {\rm in}\,\,\,\,\r\setminus\{0\}.\ee 

\item For $a\in(0,2\si),$ 
\begin{align}\la{a22} \int_{0}^{\infty}\!\!\left(e^{a x}-e^{\si x}\right) \left( \frac{|\U'(x)|^{2}}{2}+W(\U(x))\right) {\rm d}x&=\frac{a-\si}{a}\int_{0}^{\infty}\!\!\left(e^{a x}-1\right)  |\U'(x)|^{2}  {\rm d}x\\
&\le  \frac{2\max\{a-\si,\,0\}}{2\si-a} \left(\g(\si)+\frac{m}{\si}\right).\la{a6}
\end{align} 
\item 
For every $a>0,$
\begin{align} &\int_{-\infty}^{0}\left(e^{a x}-e^{\si x}\right) \left( \frac{1}{2}|\U'(x)|^{2}+W(\U(x))\right) {\rm d}x\nonumber\\
&\qquad\quad=\frac{a-\si}{a\si} \left( \frac{1}{2} |\U'(0-)|^{2}+\si\int_{-\infty}^{0}e^{ax} |\U'(x)|^{2}{\rm d}x\right)\la{a8}\\   
&\qquad\quad\le \frac{\max\{a-\si,\,0\}}{a\si}m. \la{a9}
\end{align}
\end{enumerate}
 \end{lem}  
 \begin{proof}    (1).   Since $\U''+\si\U'-\na W(\U)=0$ in $(-\infty,0),$ taking the dot product  of this  equation with $\U'$, we see that the differential equation \eqref{a1} holds for $x< 0.$ To  show that it holds for $x> 0,$ we use the following variation technique. 
 
 Let $\varphi\in C_{0}^{\infty}((0,\infty);\r).$ Set $\varphi(x)=0$ for $x\le 0.$ There exists   $t_{0}\in (0,1)$ such that 
 $$1+t \varphi'(x)> 0,\quad x+t \varphi(x)\ge 0\quad {\rm for\,\,any}\,\,x\ge0\,\,{\rm and }\,\,t\in (-t_{0},t_{0}).$$ 
Set $\V^{t}(x)=\U(x+t\varphi(x))$.   Then  $\V^{t}\in \mathcal{A}.$  Using the substitution  $y=x+t\varphi(x)$, we compute 
 \bnn\ba   J(\si,\V^{t})&=\int_{\r}e^{\si x}\left(\frac{(1+t\varphi'(x))^{2}}{2}|\U'(x+t\varphi(x))|^{2}+W(\U(x+t\varphi(x)))\right){\rm d}x\\
 &=\int_{\r}e^{\si (y-t\varphi(x))}\left(\frac{1+t\varphi'(x)}{2}|\U'(y)|^{2}+\frac{W(\U(y))}{1+t\varphi'(x)}\right){\rm d}y.\ea \enn
   Hence, 
    \bnn\ba  0&=\frac{{\rm d}}{{\rm d}t} J(\si,\V^{t})\Big|_{t=0}\\
 &=\int_{\r}e^{\si y}\left[-\si \varphi(y)\left(\frac{|\U'(y)|^{2} }{2}+ W(\U(y)) \right)+\varphi'(y)\left(\frac{|\U'(y)|^{2} }{2} - W(\U(y)) \right)\right]{\rm d}y.\ea \enn
 Set $\varphi(x)=e^{(a-\si)x}\psi(x)$ with $\psi\in C_{0}^{\infty}((0,\infty);\r)$ and $\psi\equiv 0$ on $(-\infty,0]$.  We obtain
 \bnn 0=\int_{\r}e^{ax}\left[\left(\frac{|\U'(x)|^{2}}{2}-W(\U(x))\right)\psi'(x)-\left(\frac{2\si-a}{2}|\U'(x)|^{2}+aW(\U(x))\right)\psi(x)\right]{\rm d}x.\enn
 This implies that,  in the  distribution sense, 
  \be\la{a3}\ba&\left(  \left(\frac{|\U'(x)|^{2}}{2} -W(\U(x)) \right)e^{a x}\right)'+\left(\frac{2\si-a}{2}|\U'(x)|^{2} +aW(\U(x)) \right)e^{a x}=0.\ea\ee
  Thus, we obtain \eqref{a1} by taking $a=0$ in above equation. 

Set  $\xi=\frac{1}{2}|\U'|^{2} - W(\U).$ We see that $\xi'+2\si\xi=-2 \si W(\U)$ in $(0,+\infty).$
Since  $\U\in H_{{\rm loc}}^{1}$,  we have   $W(\U)\in C^{\frac{1}{2}}.$ This implies  $\xi\in C^{\frac{3}{2}}\left([0+,+\infty)\right).$ 

   (2). 
Observe from Lemma \ref{lem7.1} that,  as  $x\rightarrow+\infty$, $|\U'(x)|^{2} +W(\U(x))=O(1) e^{-2\lambda x}$ for some  $\lambda>\si$.   For $a\in (0,2\si),$ integrating  \eqref{a3} over $(0,+\infty)$ and dividing the resulting equation  by $a$,  we obtain    \be\la{a4}\ba& \int_{0}^{\infty}\left(\frac{2\si-a }{2a}|\U'(x)|^{2}+W(\U(x)) \right)e^{a x}{\rm d}x=\frac{|\U'(0+)|^{2}}{2a}.\ea\ee
In particular,  taking  $a=\si$, we obtain 
  \be\la{a13}\ba& \int_{0}^{\infty}\left(\frac{1}{2}|\U'(x)|^{2}+W(\U(x)) \right)e^{\si x}{\rm d}x=\frac{1}{2\si}|\U'(0+)|^{2}.\ea\ee
Taking the difference of the above two equations,  we obtain 
  \begin{align}\la{a5}& \int_{0}^{\infty}\left(\frac{1}{2}|\U'(x)|^{2}+W(\U(x)) \right)\left(e^{a x}-e^{\si x}\right){\rm d}x\nonumber\\ &\qquad=\frac{a-\si}{a\si}\left(-\frac{1}{2}|\U'(0+)|^{2}+ \si\int_{0}^{\infty}e^{a x}|\U'(x)|^{2}{\rm d}x\right).
 \end{align}
 Next,  integrating \eqref{a1}  over $(0,+\infty)$ leads to 
 \be\la{a15}\frac{1}{2}|\U'(0+)|^{2}=\si\int_{0}^{\infty}|\U'(x)|^{2}{\rm d}x.\ee
 Substituting this into \eqref{a5}, we  conclude \eqref{a22}.

 Finally, from  \eqref{a4} and \eqref{a15} we see that,  for $a\in (0,2\si),$
   \bnn\ba  
   \int_{0}^{\infty}e^{a x} |\U'(x)|^{2} {\rm d}x\le \frac{|\U'(0+)|^{2}}{2\si-a}
   =\frac{2\si}{2\si-a}   \int_{0}^{\infty} |\U'(x)|^{2} {\rm d}x.
   \ea\enn
   Hence, for $a\in [0,2\si)$,
     \bnn \ba  \int_{0}^{\infty}\left(e^{a x}-1\right)|\U'(x)|^{2}{\rm d}x\le \frac{a}{2\si-a} \int_{0}^{\infty} |\U'(x)|^{2} {\rm d}x\le  \frac{2a}{2\si-a}\left(\g(\si)+\frac{m}{\si}\right),
   \ea\enn
   which give the desired \eqref{a6} by substituting it into \eqref{a22}.

   (3).  Similarly,  using $\U''+\si\U'-\na W(\U)=0$  in $(-\infty,0)$, we see that \eqref{a3} holds for any $a>0$ and $x< 0.$  Hence, integrating \eqref{a3} over  $(-\infty,0)$  gives
      \bnn\la{a7}\ba& \int_{-\infty}^{0}\left(\frac{2\si-a}{2a}|\U'(x)|^{2} +W(\U(x)) \right)e^{a x}{\rm d}x=-\frac{|\U'(0-)|^{2}}{2a}.\ea\enn
      Setting $a=\si$ gives       
  \be\la{a14}\ba& \int_{-\infty}^{0}\left(\frac{|\U'(x)|^{2}}{2}+W(\U(x)) \right)e^{\si x}{\rm d}x=-\frac{1}{2\si}|\U'(0-)|^{2}.\ea\ee
 Taking the difference of the above two equations,  we obtain  \eqref{a8}.   
 
 Finally, integrating  \eqref{a1}  over $(-\infty,0)$  gives 
 \bnn\frac{|\U'(0-)|^{2}}{2}+\si \int_{-\infty}^{0}| \U'(x)|^{2} {\rm d}x=-\lim_{x\rightarrow-\infty}W(\u(x))\le m.\enn
 Substituting this into \eqref{a8}, we get \eqref{a9}.
\end{proof}

\begin{remark} It follows   from \eqref{a13} and \eqref{a14}  that  \bnn \g(\si)=J(\si,\U)=\frac{1}{2\si} \left(|\U'(0+)|^{2}-|\U'(0-)|^{2}\right).\enn
Thus, $\U'(0+)=\U'(0-)$ only if $\g(\si)=0.$
Later, we shall show that $\U$ solves \eqref{6.8}, as long as $\g(\si)=0.$ \end{remark}

\subsection{Monotonicity and Lipschitz Continuity of $\g$}

Now we are ready to prove the following 
  \begin{thm}\la{lem3.3}\la{th2.1}  Let    $\g(\si)$  be  as defined in \eqref{a23}. The following holds:
  \begin{enumerate}
\item  if $0<a<\si,$ then $\g(a)<\g(\si);$
  
 \item $\g(\cdot)$ is Lipschitz continuous in $(0,+\infty);$  
  
\item there exist the limits 
\bnn \lim_{\si\searrow0}\g(\si)=-\infty\quad {\rm and}\quad \lim_{\si\nearrow+\infty}\g(\si)=+\infty;\enn

\item there exists a unique $\si^{*}>0$ such that $\g(\si^{*})=0.$

\end{enumerate}
  \end{thm}
  
\begin{proof}  Let $a$ and $\si$ be positive constants. 
Let $\U$ be a minimizer of $J(\si,\cdot)$ in $\mathcal{A}$. Then,
\bnn \g(a)=\inf_{\mathcal{A}}J(a,\cdot)\le J(a,\U)=\g(\si)-J(\si,\U)+J(a,\U).\enn
Using the definition of $J$, we obtain 
\be\la{a10}\ba \g(a)-\g(\si)\le  \int_{\r}\left(e^{ax}-e^{\si x}\right)\left(\frac{1}{2}|\U'(x)|^{2}+W(\U(x)) \right){\rm d}x.\ea\ee

(1). When $0<a<\si,$   it follows  from \eqref{a10}, \eqref{a22},  and \eqref{a8} that 
$
 \g(a)-\g(\si)<0.$

(2). 
Assume that $a\in (\si,2\si)$. 
Then,  by \eqref{a6}, \eqref{a9} and \eqref{a10},  we derive that 
\bnn\ba 0<\g(a)-\g(\si)
&\le \frac{a-\si}{2\si-a}\left(\g(\si)+\frac{m}{\si}\right)+\frac{a-\si}{a\si}   m.\ea\enn
Thus, $\g(\cdot)$ is Lipschitz continuous on $(0,+\infty).$

(3).  We obtain the limits  from \eqref{5.10} in Lemma \ref{lem3.2}.

(4). It directly follows from  (3), and  the monotonicity  and continuity of $\g(\cdot)$.

The proof of Theorem \ref{th2.1}  is completed.
 \end{proof}
  
\section{Proof of Theorem \ref{thm}}

By  Theorem \ref{th2.1},  there exists a unique $\si^{*}>0$ such that $\g(\si^{*})=0.$  Let   $\U$ be a minimizer of $J(\si^{*},\cdot)$ in $\mathcal{A}.$   Let $\varphi\in C_{0}^{\infty}(\r;\rr)$ be  arbitrary.  For every $t\in (-1,1)$ we define  $z(t):=\max\{z\in \r\,:\, \U(z)+t\varphi(z)\in \Ga\}.$ 
Set  $\u^{t}(x)=\u(x+z(t))+t \varphi(x+z(t))$.  Then,   $\u^{t}(0)\in \Ga$  and  $W(\u^{t}(x))\ge 0$ for all $x\ge 0.$ Thus  $\u^{t}\in \mathcal{A}.$ 
Simple computation gives 
\bnn\ba J(\si^{*},\,\u^{t})&=\int_{\r} e^{\si^{*}x}\left(\frac{1}{2}|(\u +t\varphi)|^{2}+W((\u +t\varphi))\right)(x+z(t)){\rm d}x\\
&= e^{-\si^{*}z(t)}\int_{\r} e^{\si^{*}x}\left(\frac{1}{2}|(\u +t \varphi)'|^{2}+W((\u +t\varphi))\right)(x){\rm d}x\\
&=e^{-\si^{*}z(t)} J(\si^{*},\,\u+t\varphi).\ea\enn 
Notice that $$e^{-\si^{*}z(t)} J(\si^{*},\,\u+t\varphi)=J(\si^{*},\u^{t})\ge J(\si^{*},\u)=0.$$
Hence, \bnn\la{6.9}\ba 0&\le\lim_{t\rightarrow0} \frac{ e^{\si^{*}z(t)} J(\si^{*},\,\u^{t}) }{|t|}=\lim_{t\rightarrow0}\frac{J(\si^{*},\,\u+t \varphi)-J(\si^{*},\,\u)}{|t|}.\ea\enn
Thus,
\bnn\la{6.10}\ba 0 &= \lim_{t\rightarrow0}\frac{J(\si^{*},\,\u+t \varphi)-J(\si^{*},\,\u)}{t}=:\left\langle \frac{\de J}{\de \u},\,\varphi\right\rangle\\
&=\int_{\r} e^{\si^{*}x}\left( \u' \cdot\varphi'+\na W(\u)\cdot\varphi\right){\rm d}x\quad \forall \varphi\in C_{0}^{\infty}(\r;\rr).
\ea\enn
This implies that $\U$  is a weak solution of  $\left(e^{\si^{*}x} \u'\right)'+e^{\si^{*}x}\na W(\u)=0$ in $\r$.  By a standard regularity theory we see  that $\U$ is a classical solution of   $\si\U'+\u''-\na W(\u)=0$ in $\r.$
 
Finally,  if  $\mathbb{E}=\{\a\}$, then  we conclude  from \eqref{4.19} that 
\bnn \lim_{x\rightarrow-\infty}\U(x)=\a.\enn   
This completes the proof of Theorem \ref{thm}.

\bigskip

\begin{thm}\la{thm4.1} The speed $\si^{*}$ of the positive root of $\g(\cdot)=0$  is the largest wave speed $\si$, among all solutions of  \eqref{6.8}.  Also, it satisfies \eqref{cm1}.  \end{thm}
\begin{proof}   Suppose $(\si,\U)$ is a solution of \eqref{6.8}.  
 Let $a\in (0,2c).$      By  \eqref{6.24}  and integration by parts,   we obtain 
\bnn\la{d1}\ba J(a,\U)&=\int_{\r}\left(\frac{1}{2}|\U'(x)|^{2}+W(\U(x))\right){\rm d}\frac{e^{ax}}{a}=-\int_{\r}\frac{e^{ax}}{a}\left(\U'\cdot\U''+\na W(\U)\cdot\U'\right){\rm d}x.\ea\enn
Substituting $\na W(\u)$ by $\si\U'+\U''$, and using integration by parts,  we  obtain 
\bnn\ba J(a,\U)&
=\frac{a-\si}{a}\int_{\r}e^{ax} |\U'|^{2} {\rm d}x\qquad \forall a\in (0,2c).\ea\enn
Hence, $\g(a)=\inf_{\mathcal{A}}J(a,\cdot)< 0$  if  $a\in (0,\si)$. Also,  $\g(\si)\le J(\si,\u)=0$.   Finally,  by the monotonicity of $\g(\cdot)$, we conclude $\si\le \si^{*}.$ Therefore, $\si^{*}$ is the largest speed.
 \end{proof}

 \section{An Example}
  We consider the two dimensional case for  an illustration.  
  
Let  $f(s,\si)$ for $s\in \r$ and $\si\in\r$ be defined by 
  \bnn  f(s,\si)=(s^{2}-1)(2s-\si).\enn
  We  define,   for   $\alpha\in (0,2)$ and $\beta\in [\alpha,2)$, 
 \be\la{b6}\U=(u,v),\quad  W(\U)=\int_{1}^{u}f(s,\alpha){\rm d}s+\int_{1}^{v} f(s,\beta){\rm d}s.\ee
 It easy to check that the $W$ in  \eqref{b6} satisfies the assumption ${\bf (A)}$ and admits  four {\it wells}, given by 
 \bnn \b=(1,1),\quad \a_{1}=(-1,1),\quad \a_{2}=(1,-1),\quad \a_{3}=(-1,-1).\enn
 The depths of the wells satisfy 
 $0=W(\b)>W(\a_{1})\ge W(\a_{2})>W(\a_{3});$
  the rest critical points, e.g., $(\frac{\alpha}{2},\frac{\beta}{2}),$ are non-local minima of $W$. 
The system of  traveling wave equations can be written as 
 \be\la{d6}\left\{\ba &\si u'+u''=(u^{2}-1)(2u-\alpha)\quad {\rm in}\,\,\r,\\
 &\si v'+v''=(v^{2}-1)(2v-\beta)\quad \,\,{\rm in}\,\,\r,\\
 &u(+\infty)=1,\,\,v(+\infty)=1.\ea\right.\ee
We have the following assertions:

\begin{itemize}
\item The case of  $0<\alpha<\beta<2.$ 

There exists a traveling wave $(\si, \U)$, unique up to a translation, that  connects $\a_{1}$ to $\b.$ In particular, the general solution is  given by 
\bnn \si=\alpha\quad {\rm and} \quad \U_{1}(x)=(u(x),v(x))=({\rm tanh} (x+x_{0}),\,1),\enn
where $x_{0}\in \r$ is an arbitrary constant.

There exists a traveling wave $(\si, \U)$, unique up to a translation, that  connects  $\a_{2}$ to $\b$ (or  $\a_{3}$ to $\a_{1}$).   The general solution is  given by 
\bnn\ba \si=\beta\quad& {\rm and} \quad \U_{2}(x)=(u(x),v(x))=(1,\,{\rm tanh}  (x+x_{0}))\\
&\quad ({\rm or}\quad  \U_{3}(x)=(u(x),v(x))=(-1,\,{\rm tanh}  (x+x_{0}))).\ea\enn

There {\bf does not} exist a  traveling wave   connecting $\a_{3}$ to $\b.$    Indeed,  if $(c,\,(u,v))$ is a traveling wave connecting  $\a_{3}$ to $\b,$ then $u$ solves the first equation in \eqref{d6} with  ``boundary condition'' $u(\pm\infty)=\pm1.$  So, $\si=\alpha$ and 
$u(x)=\tanh (x+x_{0})$ for some  $x_{0}\in \r;$ see the   reference \cite{fm} for the details.   Similarly,   $v$ solves the second  equation in \eqref{d6} with ``boundary condition''  $v(\pm\infty)=\pm1.$ This implies that  $\si=\beta$ and 
$v(x)=\tanh (x+y_{0})$ for some  $y_{0}\in \r.$   However, since  $\alpha\neq\beta$, we see that there is no traveling waves  connecting  $\a_{3}$ to $\b.$

\item The case of  $0<\alpha=\beta<2.$ 

There exists a unique speed and infinitely many  traveling wave trajectories   that    connect  $\a_{3}$ to $\b.$ In particular, the general solution is  given by, for any $x_{0},\,y_{0}\in \r,$ 
\bnn \si=\alpha\quad {\rm and} \quad \U(x)=\left({\rm tanh} (x+x_{0}),\,{\rm tanh} (x+y_{0})\right).\enn
\end{itemize}

\bigskip

\begin{remark} By solving the decoupled  system \eqref{d6},   we see that $\si^{*}=\beta$. In particular,  when $0<\alpha<\beta<2,$
\bnn 0=\g(\si^{*})=\g(\beta)=J(\beta,\U_{2})=J(\alpha,\U_{1})>\inf_{\mathcal{A}}J(\alpha,\cdot)=\g(\alpha).\enn
\end{remark}
\begin{remark}  In the scalar case, there is the celebrated result of Fife-McLeod \cite{fm}, which states that if there exists a traveling wave of speed $\si_{1}>0$ connecting $\a_{1}$  to $\b$ and a traveling wave of speed $\si_{31}>\si_{1}$ connecting $\a_{3}$ to $\a_{1}$, then there exists a traveling wave connecting $\a_{3}$ to $\b$.  However,  the above example shows that such a  conclusion is not true for vector valued systems.   \end{remark} 

\bigskip

{\bf Acknowledgements.} The work is supported by 
 Sichuan Science and Technology Program (Grant No.2023ZYD0003).

\end{document}